\documentclass{amsart}

\usepackage{fancyhdr}
\usepackage[english]{babel}
\usepackage[english]{babel}

\usepackage{amssymb}
\usepackage{amsfonts}
\usepackage{amsmath,amssymb}
\usepackage{amsthm}

\usepackage{mathrsfs}
\usepackage{enumerate}

\newtheorem{theorem}{Theorem}[section]

\newtheorem{proposition}{Proposition}

\theoremstyle{definition}
\newtheorem{definition}[theorem]{Definition}
\newtheorem{remark}{Remark}

\newcommand{\ds}{\displaystyle}
\newcommand{\vp}{\varphi}
\newcommand{\R}{\mathbb{R}}
\newcommand{\N}{\mathbb{N}}

\title[A fractional eigenvalue problem in $\mathbb{R}^N$] %Use the shortened version of the full title
      {A fractional eigenvalue problem in $\mathbb{R}^N$}

\author[Giacomo Bocerani and Dimitri Mugnai]{}

\subjclass{Primary: 35R11, 45C05; Secondary: 35A15, 35P15, 49R05.}
 \keywords{Fractional Laplacian, entire solution, eigenvalue problem, asymptotically linear problem.}

 \email{giacomo.bocerani@unipg.it}
 \email{dimitri.mugnai@unipg.it}

\thanks{The second author is a member of the Gruppo Nazionale per l'Analisi Matematica, la Probabilit\`a e le loro Applicazioni (GNAMPA) of the Istituto Nazionale di Alta Matematica (INdAM), and
is supported by the GNAMPA Project {\sl Systems with irregular operators}.}

\begin{document}
\maketitle

% Enter the first author's name and address:
\centerline{\scshape Giacomo Bocerani}
\medskip
{\footnotesize
% please put the address of the first author
 \centerline{Department of Mathematics and Computer Sciences, University of Perugia
}
   \centerline{Via Vanvitelli 1,}
   \centerline{06123
Perugia - Italy}
} % Do not forget to end the {\footnotesize by the sign }

\medskip

\centerline{\scshape Dimitri Mugnai}
\medskip
{\footnotesize
 % please put the address of the second  and third author
 \centerline{Department of Mathematics and Computer Sciences, University of Perugia
}
   \centerline{Via Vanvitelli 1,}
   \centerline{06123
Perugia - Italy}
}

\bigskip

%The abstract of your paper
\begin{abstract}
We prove that a linear fractional operator with an asymptotically constant lower order term in the whole space admits eigenvalues.
\end{abstract}

%The title of your section 1
\section{Introduction}

It is well known that the spectrum of the Laplace operator in $\R^N$ is purely continuous, that is $\sigma(-\Delta)=[0,\infty)$. On the other hand, considering operators of the form
\[
-\Delta+g(x) \mbox{ or }-\frac{\Delta}{g(x)},
\]
where $g:\R^N\to \R$ satisfies suitable growth conditions, one may hope to apply the standard approach in Hilbert spaces and prove the existence of a principal eigenvalue (see \cite{bcf} and \cite{DH}, \cite{LY}, \cite{MuPalogistic}, \cite{MP} for examples in a Banach setting) or of a diverging sequence of eigenvalues (see \cite{alle}, or \cite{SW} when $g(x)\to 0$ at infinity).

In the recent paper \cite{llg}, the authors consider operators of the form
\[
 -\Delta   u + \beta g(x)u,
\]
where $g\approx1$ at infinity, and study the associated spectrum. Of course, in this case the situation is different, since no compactness argument can be invoked, but they prove that eigenvalues do exist. Inspired by their result, we consider a related situation for the following eigenvalue fractional Laplacian problem in $\mathbb{R}^N$:
\begin{equation}
\label{E1}
( -\Delta )^s u + \beta g(x)u = \lambda u, \ \ x \in \mathbb{R}^N.
\end{equation}
Here $N>2s$, $s\in(0,1)$, $\beta > 0$ is a parameter and $(-\Delta)^s$ denotes the fractional Laplacian defined through the Fourier transform in the following way: for any $f\in {\mathscr S}(\R^N)$ with
Fourier transform ${\mathcal F} f=\hat{f}$, we define, modulo a positive multiplicative constant depending on $N$ and $s$,
\begin{equation}\label{deltas}
(-\Delta)^s f={\mathcal F}^{-1}\big(|\xi|^{2s}\hat{f}(\xi)\big),
\end{equation}
see \cite{va}. When $u$ is sufficiently regular, a pointwise expression of the fractional Laplacian is also available, namely
\[
( -\Delta ^s ) u :=-C_{N,s}\int_{\R^N}\frac{u(x+y)+u(x-y)-2u(x)}{|y|^{N+2s}}\,dy,
\]
for some $C_{N,s}>0$. Here, we will not go into further details about the functional setting of the problem, postponing these aspects to Section \ref{secfunctional}. We only remark that we look for couples $(\lambda,u)$ with $\lambda\in \R$ and $u\neq0$ which satisfy \eqref{E1}.

As in \cite{llg}, we assume to deal with a function $g$ which is not constant and tends to 1 at infinity. More precisely, denoting by $m(S)$ the measure of a set $S\subset \R^N$, we make the following assumptions on $g$:\\

\noindent $(g_{1})$ $g \in L^{\infty}(\mathbb{R}^N)$, $0 \le g \le 1$, $\ds\lim_{|x| \to \infty}g(x)=1$ and $m\left(\big\{ x \in \mathbb{R}^N \,:\,  g(x)<1 \big\}\right)>0$.

Although the assumptions on $g$ are rather weak, we show that also in this case the operator 
\[
L_{\beta }:H^{s}(\mathbb{R}^N) \to \left ( H^{s}(\mathbb{R}^N) \right )' \mbox{ (cfr. Section \ref{secfunctional})} , \quad u\mapsto ( -\Delta ^s ) u + \beta g(x)u,
\]
admits eigenvalues, which are all below $\beta$, see Theorem \ref{T1} below. We emphasize the fact that, in spite of this difference with the case of $-\Delta$ in bounded domains, we prove the existence of a first eigenfunction which preserves the same nice properties of the first eigenfunction of $-\Delta$ in bounded domains. More precisely, we prove that:
\begin{itemize}
\item the first eigenvalue $\lambda_1$ is simple,
\item the associated eigenfunction is strictly positive in $\R^N$,
\item any eigenfunction associated to any other eigenvalue is nodal, i.e. sign--changing,
\item there exists a different minimax characterization of the second eigenvalue.
\end{itemize}
Of course, as usual when eigenvalues are found via the Krasnoselskii genus, as we do, we don't know if other eigenvalues can show up.

Completing the previous description, in analogy with the case of $-\Delta$ in $\R^N$, we observe the following property, which is well--known, and which we prove using an elegant Pohozaev identity.
\begin{proposition}\label{noeig}
$(-\Delta)^s$ has no eigenvalues in $\R^N$, i.e. if $\mu \in \mathbb{R}$ and $u\in H^s(\R^N)$ are such that 
\begin{equation}\label{E26}
\int_{\mathbb{R}^N \times \mathbb{R}^N} 
\frac{(u(x)-u(y))(\varphi(x)-\varphi(y))}{|x-y|^{N+2s}} 
\, dxdy= \mu \int_{\mathbb{R}^N}u\varphi \, dx
\end{equation}
for all $\varphi $ in $\mathbb{R}^N$, then
$u\equiv 0$.
\end{proposition}
\begin{remark}
Of course, we reduce to problem \eqref{E26} when $g=1$ in $\R^N$, and this shows that the condition  ``$m\left(\big\{ x \in \mathbb{R}^N \,:\,  g(x)<1 \big\}\right)>0$'' in $(g_1)$ is necessary and sufficient  to get eigenfunctions for $L_\beta$.
\end{remark}

The spectral properties that we prove here are the starting point of further investigations about existence results for nonlinear fractional problems in the whole of $\R^N$, see \cite{bm}.

\section{Functional setting and eigenvalues}\label{secfunctional}

As usual, we are interested in functions satisfying \eqref{E1} in a weak sense. For this, let us recall some definitions: for $s\in (0,1)$ we denote by $H^s(\R^N)$ the Sobolev space of fractional order $s$ defined as
\[
H^s(\R^N)=\left\{u\in L^2(\R^N)\,:\,  \frac{|u(x)-u(y)|}{|x-y|^{N/2+s}}\in L^2\left(\R^N\times \R^N\right) \right\},
\]
and
\[
[u]_{H^s(\R^N)}:=\left( \int_{\R^N\times\R^N}  \frac{|u(x)-u(y)|^2}{|x-y|^{N+2s}}dxdy\right)^{1/2}
\]
denotes the Gagliardo seminorm of $u$. Moreover, we set $\|\cdot\|_2=\|\cdot\|_{L^2(\R^N)}$.

Now, by the Plancharel Theorem, $\|u\|_2=\|\hat{u}\|_2$, and by \eqref{deltas}
\begin{equation}\label{nor1}
\int_{\R^N} |(-\Delta)^{\frac{s}{2}}u|^2dx=\int_{\R^N}|\xi|^{2s}|\hat{u}(\xi)|^2d\xi
\end{equation}
for every $u\in H^s(\R^N)$. As a consequence, $\beta$ being positive, $H^s(\R^N)$ can be endowed with the norm
\[
\|u\|^2:=\int_{\R^N} |(-\Delta)^{\frac{s}{2}}u|^2dx+\beta\int_{\R^N}u^2dx
\]
for all $u\in H^s(\R^N)$, see \cite{CW}. 

On the other hand, by \cite[Lemma 3.1]{fls},
\begin{equation}\label{nor2}
\int_{\R^N} |\xi|^{2s}|\hat{u}(\xi)|^2d\xi=c_{N,s}\int_{\R^N\times\R^N}\frac{|u(x)-u(y)|^2}{|x-y|^{N+2s}}dxdy
\end{equation}
for every $u\in H^s(\R^N)$ and some positive constant $c_{N,s}$. Then, by \eqref{nor1} and \eqref{nor2}, after setting $c_{N,s}=1$, we immediately get
\begin{equation}\label{formadeb}
\int_{\R^N} (-\Delta)^{\frac{s}{2}}u (-\Delta)^{\frac{s}{2}}v\,dx=\int_{\R^N\times\R^N}\frac{\big(u(x)-u(y)\big)\big(v(x)-v(y)\big)}{|x-y|^{N+2s}}dxdy.
\end{equation}

In light of the previous considerations, we are now ready to give the following
\begin{definition}
A function $u\in H^s(\R^N)$ is an eigenfunction of problem \eqref{E1} with associated eigenvalue $\lambda$  if $u\neq 0$ and
\[
\int_{\R^N} (-\Delta)^{\frac{s}{2}}u (-\Delta)^{\frac{s}{2}}\vp\,dx+\beta \int_{\R^N}g(x)u\vp\,dx=\lambda  \int_{\R^N}u\vp\,dx
\]
for every $\vp\in H^s(\R^N)$.
\end{definition}

Of course, in view of \eqref{formadeb}, the previous identity can be written as
\begin{equation}\label{E2} 
\begin{aligned}
&\int_{\mathbb{R}^N \times \mathbb{R}^N} 
\frac{(u(x)-u(y))( \varphi (x) -  \varphi (y))}{|x-y|^{N+2s}} \, dxdy
\\  &+ \beta \int_{\mathbb{R}^N}g(x)u\varphi \, dx  = \lambda \int_{\mathbb{R}^N}u\varphi \, dx
\quad \forall \,\varphi \in H^{s}(\mathbb{R}^N) .
\end{aligned}
\end{equation}

In order to state our main result, let us introduce some notions. First, we set
$$\Sigma = \left \{ u \in H^{s}(\mathbb{R}^N) \,:\,\int_{\mathbb{R}^N}u^2 \,dx  =1\right \}
$$
and
$$\lambda _{1}= \inf _{u \in \Sigma } \Phi (u),
$$
where 
$$\Phi (u)= \int_{\mathbb{R}^N \times \mathbb{R}^N} \frac{|u(x)-u(y)|^{2}}{|x-y|^{N+2s}} \, dxdy + 
\beta \int_{\mathbb{R}^N}g(x)u^{2} \, dx .$$
Then, for every $k \in \mathbb{N}$, we introduce the families of sets 
$$
\Sigma _{k}= \Big\{ A \subset \Sigma _{1} \,:\, A \ \mbox{ is  compact, } -A=A \mbox{ and }  \gamma (A) \ge k  \Big\}, $$
where $\gamma (A)$ denotes the Krasnoselskii genus of a symmetric set A, defined as 
$$\gamma (A) =\left\{
\begin{array}{l}
\inf \Big\{ m \,:\, \exists \, h \in C^0\left ( A,\mathbb{R}^m \right) \mbox{ with } h(u)=-h(-u)  \Big\},
\\ \infty , \ \mbox{if} \ \left \{...\right \} = \emptyset , \ \mbox{ in particular  if} \ 0 \in A,
\end{array}
\right. $$
see \cite{Struwe}. Associatd to $\Sigma_k$ we set
$$
\lambda _{k}= \inf _{A \in \Sigma _{k}} \sup _{u \in A}\Phi (u).
$$
Finally, we define
$$M= \left \{ u \in H^{s}(\mathbb{R}^N) \,:\,\int _{\mathbb{R}^N} (1-g(x))|u|^2 \, dx = 1  \right \},$$
$$M_{k}= \Big\{ A \subset M \,:\,  A \mbox{ is compact, } -A=A \mbox{ and } 
\gamma (A) \ge k \Big\},  \ \  k \in \mathbb{N},$$
and 
$$ \Gamma _{k} = \inf _{A \in M_{k}} \sup _{u \in A}
\int_{\mathbb{R}^N \times \mathbb{R}^N} \frac{|u(x)-u(y)|^{2}}{|x-y|^{N+2s}} \, dxdy, \ \ \  k \in \mathbb{N}.$$
Note that $M\not = \emptyset$ by $(g_1)$.

Now, we are ready to state our main result:
\begin{theorem}\label{T1}
Assume $(g_{1})$. Then
\begin{itemize}
\item[$(1)$] $\left \{ \lambda _{k} \right \} _{k}$ is a nondecreasing sequence and $\lambda _{k} \le \beta$ for all $k\in \N$.
\item[$(2)$] If $\lambda _{k}< \beta $, then $\lambda _{k}$ is an eigenvalue of the operator $L_{\beta }$, namely there exist an eigenfunction $u_{k} \in \Sigma _{k}$ such that for any $\varphi \in H^{s}(\mathbb{R}^N)$
\begin{equation}
\label{E3}
\begin{aligned}
&\int_{\mathbb{R}^N \times \mathbb{R}^N} 
\frac{(u_k(x)-u_k(y))( \varphi (x) -  \varphi (y))}{|x-y|^{N+2s}} \, dxdy
\\  &+ \beta \int_{\mathbb{R}^N}g(x)u_k\varphi \, dx  = \lambda _{k}\int_{\mathbb{R}^N}u_k\varphi \, dx.
\end{aligned}
\end{equation}
\item[$(3)$] If $ \lambda _{1}$  is an eigenvalue, then it is simple and the associated eigenfunction $e_1$  is 
strictly positive. Moreover, eigenfunctions associated to eigenvalues different from $ \lambda_{1}$  are nodal, i.e. they change sign.
\item[$(4)$] If $\Gamma _{k}<\beta$, then $\lambda _{k}< \beta $, and hence $\lambda _{k}$  is an eigenvalue.
\end{itemize}
\end{theorem}

\begin{proof}
$(1)$ Obviously, $\Sigma _{1} \supset \Sigma _{2} \supset  \ldots \supset \Sigma _{k} \supset  \ldots $, and from the very definition of $\lambda _{k}$ we have that $\left \{ \lambda _{k} \right \} _{k}$ is a nondecreasing sequence. Now, let us take 
$A \in \Sigma _{k}$, and define 
$$
A_{t}=\left \{ v_t \,:\, v_t(x)=t^{\frac{N}{2}}u(tx), u \in A, t>0 \right \}.
$$
Then, for all $v \in A_{t}$, we have
$$
\int _{\mathbb{R}^N}|v_t(x)|^2 \, dx = \int _{\mathbb{R}^N}|u(y)|^2 \, dy
$$
and 
\[
\int_{\mathbb{R}^N \times \mathbb{R}^N} \frac{|v_t(x)-v_t(y)|^{2}}{|x-y|^{N+2s}} \, dxdy 
= t^{2s}\int_{\mathbb{R}^N \times \mathbb{R}^N} \frac{|u(w)-u(z)|^{2}}{|w-z|^{N+2s}} \, dwdz.
\]
So, we have
\begin{displaymath}
\begin{aligned}
\Phi (v_t)&=\int_{\mathbb{R}^N \times \mathbb{R}^N} \frac{|v_t(x)-v_t(y)|^{2}}{|x-y|^{N+2s}} \, dxdy +
\beta \int _{\mathbb{R}^N}g(x)v_t^2 \, dx
\\ & = t^{2s}\int_{\mathbb{R}^N \times \mathbb{R}^N} \frac{|u(w)-u(z)|^{2}}{|w-z|^{N+2s}} \, dwdz + 
\beta \int _{\mathbb{R}^N}g\left (\frac{y}{t}\right )u^2 \, dy.
\end{aligned}
\end{displaymath} 
Since, from Lebesgue's dominated convergence theorem,
$$ \lim_{t \to 0^{+}} \int _{\mathbb{R}^N}g\left (\frac{y}{t}\right )u^2 \, dy= \int _{\mathbb{R}^N}u^2 \, dy=1,$$
we immediately have that
$$
\lim_{t \to 0^{+}}\Phi (v_t)= \beta.
$$
Let us notice that such a limit is uniform in $u\in A$, since $A \in \Sigma _{k}$ is compact. Moreover, also
$A_{t} \in \Sigma _{k}$ for every $t>0$, so that $\lambda _{k} \le \sup _{v \in A_{t}}\Phi (v)$, and, as a consequence,
$$ \lambda _{k} \le \lim_{t \to 0^{+}}\sup _{v \in A_{t}}\Phi (v)=\beta,
$$
and the claim holds.\\

\noindent $(2)$ Let us suppose that $\lambda _{k}< \beta $. If $\left. \Phi \right| _{\Sigma}$ satisfies the $(PS)$ condition\footnote{We recall that $\left. \Phi \right| _{\Sigma_k}$ satisfies the $(PS)$ condition if any sequence $\{u_n\}_n$ in $\Sigma_k$ such that $\left. \Phi \right| _{\Sigma_k}'(u_n) \xrightarrow[n\to \infty]{}0$ and $\{\Phi(u_n)\}_n$ is bounded, admits a convergent subsequence.}, then $\lambda _{k}$ is a critical value of $\left . \Phi \right | _{\Sigma }$ (see \cite[Theorem 10.9]{Amb-Malc}), and there exists $u\in \Sigma_k$ such that
\begin{equation}
\label{E6}
\begin{aligned}
\int_{\mathbb{R}^N \times \mathbb{R}^N} \frac{(u(x)-u(y))( \varphi (x) -  \varphi (y))}{|x-y|^{N+2s}} \, dxdy 
&+\beta \int _{\mathbb{R}^N}g(x)u\varphi \, dx \\
&=\lambda_k \int _{\mathbb{R}^N}u\varphi \, dx 
\end{aligned}
\end{equation}
for all $\varphi\in  H^{s}(\mathbb{R}^N)$.

Now, suppose that $(PS)$ does not hold. By the Ekeland Principle (see \cite[Theorem 8.5]{willem}) we can  find a $(PS)$ sequence for $\left . \Phi \right | _{\Sigma_k}$ at level $\lambda_k$: this means that there exist a sequence 
$\left \{ u_{n} \right \}_{n}
\subset \Sigma_k$ and a sequence  $\left \{ \mu _{n} \right \}_{n} \subset \mathbb{R}$ such that
\begin{equation}\label{E5}
\Phi ( u_{n})= \int_{\mathbb{R}^N \times \mathbb{R}^N} \frac{|u_{n}(x)-u_{n}(y)|^{2}}{|x-y|^{N+2s}} \, dxdy 
+  \beta \int _{\mathbb{R}^N}g(x)u_{n}^2 \, dx \to \lambda _{k}
\end{equation}
as $n \to \infty$, and
\begin{equation}\label{E4}
\begin{aligned}
\int_{\mathbb{R}^N \times \mathbb{R}^N} &\frac{(u_{n}(x)-u_{n}(y))( \varphi (x) -  \varphi (y))}{|x-y|^{N+2s}} \, dxdy 
+\beta \int _{\mathbb{R}^N}g(x)u_{n}\varphi \, dx \\& - \mu _{n} \int _{\mathbb{R}^N}u_{n}\varphi \, dx =
o(1)\left \|\varphi \right \|
\end{aligned}
\end{equation}
for all $\varphi \in  H^{s}(\mathbb{R}^N)$.

By \eqref{E5}, it is clear that $\left \{ u_{n} \right \}_{n}$ is bounded in $ H^{s}(\mathbb{R}^N)$. Thus, taking $\varphi = u_{n} $ in \eqref{E4}, we obtain
$$
\lim_{n \to \infty} \left(\int_{\mathbb{R}^N \times \mathbb{R}^N} \frac{|u_{n}(x)-u_{n}(y)|^2}{|x-y|^{N+2s}} \, dxdy 
+\beta \int _{\mathbb{R}^N}g(x)u_{n}^2\, dx - \mu _{n} \int _{\mathbb{R}^N}u_{n}^2\, dx\right) =0.
$$
This limit, together with \eqref{E5} and the fact that $\left \{ u_{n} \right \}_{n} \subset \Sigma_k$, implies that 
$$\lim_{n \to \infty} \mu _{n} = \lambda _{k}.
$$

Since $H^{s}(\mathbb{R}^N)$ is a Hilbert space, we can assume that
\begin{equation}\label{E7}
\begin{aligned}
& u_{n} \rightharpoonup u \in H^{s}(\mathbb{R}^N)\\
& u_{n}\to u \in L^q_{\rm loc}(\mathbb{R}^N)\mbox{ for all }q\in [1, 2^{s_*}),\\
&  u_{n}\to u  \mbox{ a.e in $\R^N$},
\end{aligned}
\end{equation}
where $2^{s_*}=\frac{2N}{N-2s}$ is the critical fractional Sobolev exponent (see \cite{stein}).

Now, fixed $\varphi \in C^{\infty}_{0}(\mathbb{R}^N)$. From \eqref{E7} we have
$$ \int _{\mathbb{R}^N}g(x)u_{n}\varphi\, dx \to  \int _{\mathbb{R}^N}g(x)u\varphi\, dx,
$$
$$ \mu _{n} \int _{\mathbb{R}^N}u_{n}\varphi\, dx \to \lambda_{k} \int _{\mathbb{R}^N}u\varphi\, dx $$
and
$$
\begin{aligned} &\lim_{n \to \infty} \int_{\mathbb{R}^N \times \mathbb{R}^N} \frac{(u_{n}(x)-u_{n}(y))
( \varphi (x) -  \varphi (y))}{|x-y|^{N+2s}} \, dxdy \\& \ \ = \int_{\mathbb{R}^N \times \mathbb{R}^N} 
\frac{(u(x)-u(y))( \varphi (x) -  \varphi (y))}{|x-y|^{N+2s}} \, dxdy  . 
\end{aligned}
$$
Thus \eqref{E4} implies 
$$
\begin{aligned} \int_{\mathbb{R}^N \times \mathbb{R}^N}\frac{(u(x)-u(y))( \varphi (x) -  \varphi (y))}{|x-y|^{N+2s}} 
\, dxdy &+\beta \int _{\mathbb{R}^N}g(x)u\varphi\, dx 
\\&= \lambda_{k} \int _{\mathbb{R}^N}u\varphi\, dx
\end{aligned}
$$
for all $\varphi \in C^{\infty}_{0}(\mathbb{R}^N)$, and, by density, for all $\varphi \in H^{s}(\mathbb{R}^N)$. This shows that $u$ solves \eqref{E1} with $\lambda=\lambda_k$. Now, we show that 
$u\not\equiv 0$. By the Concentration--Compactness Principle, either there exist $R>0$ and $\nu >0$ such that (up to subsequences)
\begin{equation}\label{E8}
\lim_{n \to \infty} \int_{B_{R}}u_{n}^2 \,dx=\nu >0,
\end{equation}
or
\begin{equation}\label{E9}
\lim_{n \to \infty} \int_{B_{R}}u_{n}^2 \,dx=0  \mbox{ for all $R>0$}.
\end{equation}

If (\ref{E8}) holds, then obviously $u\not\equiv 0$. Thus, suppose that (\ref{E9}) holds. By $(g_{1})$ we have
that for all $\epsilon>0$ there exist $R_{\epsilon}>0$ such that $0\leq 1-g(x)<\epsilon$ if $|x|>R_\epsilon$, and from (\ref{E9})
$$
0\leq \int_{ \mathbb{R}^N}(1-g(x))u_{n}^2 \, dx =\int_{B_{R_\epsilon}}(1-g(x))u_{n}^2 \, dx+ \int_{ \mathbb{R}^N\setminus B_{R_{\epsilon}}}(1-g(x))u_{n}^2 \, dx.
$$
Now, 
\[
0\leq \int_{B_{R_\epsilon}}(1-g(x))u_{n}^2 \, dx\leq \int_{B_{R_\epsilon}}u_n^2dx\to 0
\]
as $n\to \infty$ by assumption, while 
\[
0\leq \int_{\R^N\setminus B_{R_\epsilon}}(1-g(x))u_{n}^2 \, dx\leq  \epsilon \int_{ \mathbb{R}^N\setminus B_{R_{\epsilon}}}u_n^2 \, dx \le \epsilon,
\]
so that
\begin{equation} \label{E10} 
\lim_{n\to \infty}\int_{ \mathbb{R}^N}(1-g(x))u_{n}^2 \, dx=0.
\end{equation}

From \eqref{E10} we immediately get 
\begin{equation} \label{E11} 
\lim_{n \to \infty} \int_{ \mathbb{R}^N}g(x)u_{n}^2 \, dx=\lim_{n \to \infty}\int_{ \mathbb{R}^N}(g(x)-1)u_{n}^2 \, dx
+1=1.
\end{equation} 
Choosing $\varphi=u_{n}$ as test function in (\ref{E4}) and passing to the limit, we obtain 
$\beta =  \lambda_{k} - c$ for some constant $c\geq0$. Hence, $\beta \le \lambda_{k}$, which contradicts our assumption. 

Summing up, we have shown that $\lambda_{k}$ is an eigenvalue with associated eigenfunction $u\not\equiv 0$.\\

\noindent $(3)$ First, let us notice that we can always find a positive eigenfunction associated to $\lambda_{1}$. Indeed, for all $u\in H^s(\R^N)$, by the triangle inequality we have
$$
\begin{aligned}\Phi(u)&=\int_{\mathbb{R}^N \times \mathbb{R}^N} \frac{|u(x)-u(y)|^{2}}{|x-y|^{N+2s}} 
\, dxdy +  \beta\int_{\mathbb{R}^N}g(x)u^2\, dx 
\\ &\ge  \int_{\mathbb{R}^N \times \mathbb{R}^N} \frac{||u(x)|-|u(y)||^{2}}{|x-y|^{N+2s}}
\, dxdy+ \beta\int_{\mathbb{R}^N}g(x)u^2\, dx =\Phi(|u|).
\end{aligned}
$$
Thus, if $e_{1}$ is an eigenfunction associated to $\lambda_{1}$, and without loss of generality we can assume $|e_{1}|_{p}=1$, we have that
$$
\lambda_{1}=\Phi(e_{1})\ge \Phi(|e_{1}|)\ge\ \inf_{|u|\in \Sigma}\Phi(u)\ge \lambda_{1},
$$
so that also $|e_{1}|\ge 0$ is an eigenfunction, as well. Thus, $|e_{1}|$ is a solution of (\ref{E2}), and by the regularity result of \cite[Theorem 3.4]{Fel-Quaas-Tan}, we get that $|e_{1}| \in C^{0,\mu}(\mathbb{R}^N)$ for some $\mu \in (0,1)$. Moreover, thanks to the Harnack's inequality in \cite[Theorem 3.1]{Tan-Xiong}, we conclude that $|e_{1}|$ is strictly positive in $\R^N$. 
This implies that all signed solutions of (\ref{E2}) with $k=1$ are strictly positive (or negative) in the whole of $\R^N$.

Now let $u$ be another eigenfunction associated to $\lambda_{1}$. Then, for all $R>0$ there exists
$\chi_{R}\in \R$ such that $$\int_{B_{R}}e_1\, dx=\chi_{R}\int_{B_{R}}u\, dx,
$$ 
that is
$$
\int_{B_{R}}(e_1-\chi_{R}u)\, dx=0.
$$ 
Since $e_1-\chi_{R}u$ can be assumed to be a signed eigenfunction associated to $\lambda_{1}$, then the previous identity implies that $e_1-\chi_{R}u=0$.  Now, if $R_{1} <R_{2}$, we find $e_1=\chi_{R_{1}}u$ in $B_{R_{1}}$ and $e_1=\chi_{R_{2}}u$ in $B_{R_{2}}$. Hence, $\chi_{R_{2}}=\chi_{R_{1}}=\chi$.
This implies that $e_1=\chi u$ in the whole space $\mathbb{R}^N$, i.e. the first eigenvalue is simple.

Finally, let us suppose that $\phi $ is an eigenfunction associated to an eigenvalue $\lambda > \lambda_{1}$, and let us suppose that $\phi$ is signed, for example
$\phi$ is nonnegative, that is $\phi_{-}:=\max\{0,-\phi\}\equiv0$. We know that if $v$, $w$ are eigenfunctions associated to eigenvalues $\mu \ne \nu$, then 
$$
<L_{\beta}v,w>=\mu\int_{\mathbb{R}^N}vw\, dx \qquad <L_{\beta}w,v>=\nu\int_{\mathbb{R}^N}wv\, dx,
$$
so that, using the fact that $<L_{\beta}v,w>=<L_{\beta}w,v>$, we get
$$
\int_{\mathbb{R}^N}wv\, dx=0.
$$
In particular, we have
$$
\int_{\mathbb{R}^N}e_{1}\phi\, dx=0,
$$ 
which is absurd, because both $e_{1}$ and $\phi$ are nonnegative and non zero. Then $\phi_{-}\not\equiv0$. A similar argument can be used 
to prove that $\phi_{+}\not\equiv0$. In conclusion, $\phi$ is sign changing.
\\

\noindent$(4)$ Let us suppose that $\Gamma_{k}<\beta $. By definition of $\Gamma_{k}$, there exist a compact and  
symmetric set $A \in M_{k}$  such that 
\begin{equation} \label{E12} \sup _{u \in A}\int_{\mathbb{R}^N \times \mathbb{R}^N} \frac{|u(x)-u(y)|^{2}}
{|x-y|^{N+2s}} \, dxdy<\beta. \end{equation}
Let us define 
$$A^{*}=\left \{u^*\,:\, u^*=\frac{u}{\|u\|_{2}}, \, u\in A \right \}.
$$
Of course, 
$A^{*}\subset \Sigma$ and $A^{*}\in \Sigma_{k}$, in fact $A^*$ is compact and $\gamma(A^*)\ge k$. Now, for any $u^* \in A^*$, we have
$$
\begin{aligned}
\Phi(u^*)&=\frac{\ds\int_{\mathbb{R}^N \times \mathbb{R}^N} 
\frac{|u(x)-u(y)|^{2}}{|x-y|^{N+2s}} \, dxdy+ \beta\int_{\mathbb{R}^N}g(x)u^2\, dx } 
{\ds\int_{\mathbb{R}^N}u^2\, dx } \\
& < \frac{\ds\beta \left (1+\int_{\mathbb{R}^N}g(x)u^2\, dx \right ) } 
{\ds \int_{\mathbb{R}^N}u^2\, dx}=\beta,
\end{aligned}
$$
by \eqref{E12}, recalling that $\int_{\mathbb{R}^N}(1-g(x))u^2\, dx=1$, since $u \in A \in M_{k}$. But
$A^*$ is compact, so 
$$
\sup _{u^* \in A^*}\Phi(u^*)<\beta,
$$
and thus
$$
\lambda_{k}\le \sup _{u^* \in A^*}\Phi(u^*)<\beta.
$$
\end{proof}

We conclude this section with the
\begin{proof}[Proof of Proposition $\ref{noeig}$]
Suppose $u\in H^s(\R^N)$ is a solution of \eqref{E26}, then $u$ satisfies the Pohozaev identity 
$$
\frac{N-2s}{2}\int_{\mathbb{R}^N}|\xi |^{2s}|\hat{u}(\xi)|^2\, d\xi = N\mu \int_{\mathbb{R}^N}\frac{u^2}{2}\, dx,
$$
(see \cite[Proposition 4.1]{Secchi}). By \eqref{nor2}, we immediately get
$$
\frac{N\mu}{2} \int_{\mathbb{R}^N}u^2\, dx = 
\mu \frac{N-2s}{2}\int_{\mathbb{R}^N}u^2\, dx.
$$ If $u \not\equiv 0$ and $\mu\neq0$, then 
$$
 \frac{N\mu}{2} =\mu \frac{N-2s}{2},$$
that is $2s=0$, which is absurd. If $\mu=0$, we get $[u]_{H^s(\R^N)}=0$, so that $\hat u\equiv0$, and so $u\equiv 0$.
\end{proof}

\subsection{Another characterization for $\boldsymbol{\lambda_{2}}$}

For further investigations where the spectrum of $L_\beta$ plays a r\^ole, we think it is useful to provide another characterization of the second eigenvalue $\lambda_{2}$ in terms of minimax values:
\begin{proposition}\label{P1}
Let us set
$$
\Sigma^*=\Big\{h\,:\, h\in C\big( [0,1],\Sigma \big) \mbox{ and } -h(1)=h(0) \ge 0  \Big\}
$$
and
\[
\lambda^*=\inf_{h\in \Sigma^*} \sup_{t\in [0,1]}\Phi(h(t)).
\]
Then $\lambda^*=\lambda_{2}.$
\end{proposition}
\begin{proof}
Take $h\in \Sigma^*$ and define the map $\tilde{h}:S^{1} \to \Sigma$ by
$$
\tilde{h}(e^{i\theta}) =
\begin{cases}
h\left (\dfrac{\theta}{\pi}\right ) & \mbox{ if } 0\le \theta \le \pi\\  \\
-h\left (\dfrac{\theta}{\pi}-1\right ) & \mbox{ if } \pi \le \theta \le 2\pi.
\end{cases}
$$
Set $A=\tilde{h}(S^{1})$; then $A\in \Sigma_{2}$. In fact, $A$ is compact, because $\tilde{h}$ is continuous; moreover, $A$ is symmetric, for $\tilde{h}$ is odd; finally, $\gamma(A)\ge 2$: if we suppose by contradiction that $\gamma(A)=1$, there should exist 
an odd function $\mathscr{H} \in C^0\left (A,\mathbb{R}\setminus \left \{0\right \} \right )$, 
but $A$ is a connected set, being the image of a connected set through a continuous function, while $\mathscr{H}(A)$ is disconnected, and a contradiction arises. 

So, by definition of $\lambda_{2}$, we have that
\begin{equation}\label{E13}
\lambda_{2}\le \sup_{u\in A}\Phi(u)=\sup_{t\in[0,1]}\Phi(h(t)),
\end{equation}
and hence 
$$
\lambda_{2}\le \inf_{h\in \Sigma^*} \sup_{t\in[0,1]}\Phi(h(t)) =\lambda^*.
$$

By Theorem \ref{T1}.(1), we know that $\lambda_{2}\le \beta$, and using arguments similar to those used therein, we can prove that $\lambda^*\le \beta$. Thus, if $\lambda_{2}=\beta$, then $\lambda_{2}=\lambda^*=\beta$. 

On the other hand, if $\lambda_{2}<\beta$, by 
Theorem \ref{T1}.(2),.(3), there exist a sign-changing eigenfunction $u$, corresponding to the eigenvalue $\lambda_{2}$. 
Let $u=u_{+}-u_{-}$, with $u_{+},u_{-}\not\equiv 0$, and where $u_+:=\max\{u,0\}$. Taking $\varphi =u_{\pm}$ as test function in 
(\ref{E3}) with $k=2$, we have 
\begin{equation}\label{sopra}
\begin{aligned} &\int_{\mathbb{R}^N \times \mathbb{R}^N} 
\frac{(u(x)-u(y))(u_{\pm}(x) -  u_{\pm}(y))}{|x-y|^{N+2s}} \, dxdy
\\  &+ \beta \int_{\mathbb{R}^N}g(x)uu_{\pm}\, dx  = \lambda _{2}\int_{\mathbb{R}^N}uu_{\pm}\, dx.
\end{aligned}
\end{equation}
Writing $u=u_+-u_-$, we can observe that
$$
\begin{aligned}
&\int_{\mathbb{R}^N \times \mathbb{R}^N} 
\frac{(u(x)-u(y))(u_{+}(x) -  u_{+}(y))}{|x-y|^{N+2s}} \, dxdy  \\
& = \int_{\mathbb{R}^N \times \mathbb{R}^N} 
\frac{|u_{+}(x) -  u_{+}(y)|^2}{|x-y|^{N+2s}} \, dxdy +2\int_{\mathbb{R}^N \times \mathbb{R}^N} 
\frac{u_{+}(x) u_{-}(y)}{|x-y|^{N+2s}} \, dxdy .
\end{aligned}
$$
Using a similar argument with $u_{-}$, from \eqref{sopra} we have that 
\begin{equation}\label{E14}
\begin{aligned}
\Phi(u_{\pm})&=\int_{\mathbb{R}^N \times\mathbb{R}^N}\frac{|u_{\pm}(x) -  u_{\pm}(y)|^2}
{|x-y|^{N+2s}} \, dxdy+\beta \int_{\mathbb{R}^N}g(x)u_{\pm}^2\, dx
\\&=-2\int_{\mathbb{R}^N \times \mathbb{R}^N} 
\frac{u_{+}(x) u_{-}(y)}{|x-y|^{N+2s}} \, dxdy
+\lambda _{2}\int_{\mathbb{R}^N}u_{\pm}^2\, dx .
\end{aligned}
\end{equation}

Now, set $h:[0,1] \to \Sigma$ by
$$h(t)=\dfrac{u_{+}\cos(\pi t)+u_{-}\sin(\pi t)}{\left\| u_{+}\cos(\pi t)+u_{-}\sin(\pi t) \right\|_{2}}.
$$
Obviously, $h$ is continuous and 
$$
h(0)=\dfrac{u_{+}}{\|u_{+}\|_{2}} ,\qquad h(1)=-\dfrac{u_{+}}{\|u_{+}\|_{2}},
$$
so $h\in \Sigma^{*}$. Let us compute $\Phi(h(t))$:
$$\begin{aligned}
&\Phi(h(t))\\
&=\int_{\mathbb{R}^N \times \mathbb{R}^N} 
\frac{((u_{+}(x)\cos(\pi t)+u_{-}(x)\sin(\pi t))-((u_{+}(y)\cos(\pi t)+u_{-}(y)\sin(\pi t))^2}
{\left\| u_{+}\cos(\pi t)+u_{-}\sin(\pi t) \right\|^2_{2}|x-y|^{N+2s}} \, dxdy
\\& \ \ \ +\beta   \dfrac{\cos^2(\pi t)}{\| ... \|_2^2} \int_{\mathbb{R}^N}g(x)|u_{+}|^2\, dx
+\beta \dfrac{\sin^2(\pi t)}{\| ... \|_2^2} \int_{\mathbb{R}^N}g(x)|u_{-}|^2\, dx.
\\
& = \dfrac{\cos^2(\pi t)}{\| ... \|_2^2}\int_{\mathbb{R}^N \times \mathbb{R}^N} 
\frac{|u_{+}(x) -  u_{+}(y)|^2}{|x-y|^{N+2s}} \, dxdy+ \dfrac{\sin^2(\pi t)}{\| ... \|_2^2}
\int_{\mathbb{R}^N \times \mathbb{R}^N} \frac{|u_{-}(x) -  u_{-}(y)|^2}{|x-y|^{N+2s}} \, dxdy\\
&-\frac{4\cos(\pi t)\sin(\pi t)}{\| ... \|_2^2}\int_{\mathbb{R}^N \times \mathbb{R}^N}\frac{u_{+}(x)u_{-}(y)}{|x-y|^{N+2s}}\, dxdy +\beta   \dfrac{\cos^2(\pi t)}{\| ... \|_2^2} \int_{\mathbb{R}^N}g(x)|u_{+}|^2\, dx\\
&+\beta \dfrac{\sin^2(\pi t)}{\| ... \|_2^2} \int_{\mathbb{R}^N}g(x)|u_{-}|^2\, dx.
\end{aligned}$$
By \eqref{E14}, we get
\begin{equation}\label{fit}
\begin{aligned}
\Phi(h(t))&
=\dfrac{\cos^2(\pi t)}{\|...\|^2_2}\left(\lambda_2\int_{\R^N}u_+^2dx-2\int_{\R^N}\dfrac{u_+(x)u_-(y)}{|x-y|^{N+2s}}dx\right)\\
&+\dfrac{\sin^2(\pi t)}{\|...\|^2_2}\left(\lambda_2\int_{\R^N}u_-^2dx-2\int_{\R^N}\dfrac{u_+(x)u_-(y)}{|x-y|^{N+2s}}dx\right)\\
&-4\dfrac{\cos (\pi t)\sin (\pi t)}{\|...\|_2^2}\int_{\R^N}\dfrac{u_+(x)u_-(y)}{|x-y|^{N+2s}}dx.
\end{aligned}
\end{equation}
But
\begin{equation}\label{mag0}
u_+(x)u_-(y)\left(\cos^2(\pi t)+\sin^2(\pi t)+2\cos (\pi t)\sin (\pi t)\right]\ge0,
\end{equation}
so that \eqref{fit} and \eqref{mag0} imply that 
$$
\Phi(h(t))\le \lambda_{2}\dfrac{ \cos^2(\pi t)}{\| ... \|^2}\int_{\mathbb{R}^N}u_{+}^2\, dx
+\lambda_{2}\dfrac{ \sin^2(\pi t)}{\| ... \|^2}\int_{\mathbb{R}^N}u_{-}^2\, dx=\lambda_2.
$$
Thus $\Phi(h(t))\le \lambda_{2}$ for all $t\in[0,1]$, and then $$\lambda^*\le \sup_{t\in [0,1]}\Phi(h(t))\le \lambda_{2},$$
and hence $\lambda^*= \lambda_{2}$.
\end{proof}

\medskip
% The data information below will be filled by AIMS editorial staff
Received xxxx 20xx; revised xxxx 20xx.
\medskip

\end{document}